\documentclass[11pt,A4paper]{article}

\usepackage[left=32mm,right=32mm,top=30mm,bottom=32mm]{geometry}
\usepackage{epsfig,labelfig}
\usepackage{epstopdf}

\usepackage{caption}

\usepackage{xfrac}

\usepackage[percent]{overpic}

\usepackage{tikz}
\usetikzlibrary{positioning}
\usetikzlibrary{arrows}

\usepackage{color}
\usepackage{amsthm,amsmath,amssymb}
\usepackage{booktabs}
\usepackage{mathpazo}
\usepackage{microtype}
\usepackage{overpic}
\usepackage{bm}
\usepackage{sectsty}
\usepackage[
	pdftitle={PDFTitle},
	pdfauthor={Hugo Parlier},
	ocgcolorlinks,
	linkcolor=linkred,
	citecolor=linkred,
	urlcolor=linkblue]
{hyperref}

\usepackage{mathtools}

\definecolor{linkred}{RGB}{199,21,133}
\definecolor{linkblue}{RGB}{16, 78, 139}

\usepackage[hang,flushmargin]{footmisc}
\usepackage{enumitem}
\usepackage{titlesec}
	\titlespacing{\section}{0pt}{12pt}{0pt}
	\titlespacing{\subsection}{0pt}{6pt}{0pt}
	
\titlelabel{\thetitle.\quad}

\makeatletter 

\long\def\@footnotetext#1{%
\H@@footnotetext{%
\ifHy@nesting 
\hyper@@anchor{\@currentHref}{#1}%
\else 
\Hy@raisedlink{\hyper@@anchor{\@currentHref}{\relax}}#1%
\fi 
}}

\def\@footnotemark{%
\leavevmode 
\ifhmode\edef\@x@sf{\the\spacefactor}\nobreak\fi 
\H@refstepcounter{Hfootnote}%
\hyper@makecurrent{Hfootnote}%
\hyper@linkstart{link}{\@currentHref}%
\@makefnmark 
\hyper@linkend 
\ifhmode\spacefactor\@x@sf\fi 
\relax 
}%

\ifFN@multiplefootnote%
\renewcommand*\@footnotemark{%
\leavevmode 
\ifhmode 
\edef\@x@sf{\the\spacefactor}%
\FN@mf@check 
\nobreak 
\fi 
\H@refstepcounter{Hfootnote}%
\hyper@makecurrent{Hfootnote}%
\hyper@linkstart{link}{\@currentHref}%
\@makefnmark 
\hyper@linkend 
\ifFN@pp@towrite 
\FN@pp@writetemp 
\FN@pp@towritefalse 
\fi 
\FN@mf@prepare 
\ifhmode\spacefactor\@x@sf\fi 
\relax%
}%
\fi 

\makeatother 

\newtheorem{thm}{Theorem}[section]

\newtheorem{coro}[thm]{Corollary}
\newtheorem{lem}[thm]{Lemma}

\theoremstyle{definition}

\theoremstyle{remark}

\newtheorem{remk}[thm]{Remark}
\newtheorem{claim}[thm]{Claim}

\renewcommand{\phi}{\varphi}

\newcommand{\LA}{{\rm La}}

\newcommand{\MM}{{\mathcal M}}
\newcommand{\TT}{{\mathrm{Teich}}}

\newcommand{\be}{ \begin{equation} }
\newcommand{\ee}{ \end{equation} }
\DeclareMathOperator{\sech}{sech}

\sectionfont{\large \bfseries}
\subsectionfont{\normalsize}

\long\def\symbolfootnote[#1]#2{\begingroup%
\def\thefootnote{\fnsymbol{footnote}}\footnote[#1]{#2}\endgroup}

\def\blfootnote{\xdef\@thefnmark{}\@footnotetext}

\date{\today}

\begin{document}

{\Large \bfseries 
Measuring pants
}

{\large 
Nhat Minh Doan\symbolfootnote[1]{
Research supported by FNR PRIDE15/10949314/GSM.}, Hugo Parlier\symbolfootnote[2]{
Research partially supported by ANR/FNR project SoS, INTER/ANR/16/11554412/SoS, ANR-17-CE40-0033.} 
and Ser Peow Tan\symbolfootnote[3]{
Research partially supported by R146-000-289-114. \vspace{.1cm} 
}
}

\vspace{0.5cm}

{\bf Abstract.} 
We investigate the terms arising in an identity for hyperbolic surfaces proved by Luo and Tan, namely showing that they vary monotonically in terms of lengths and that they verify certain convexity properties. Using these properties, we deduce two results. As a first application, we show how to deduce a theorem of Thurston which states, in particular for closed hyperbolic surfaces, that if a simple length spectrum "dominates" another, then in fact the two surfaces are isometric. As a second application, we show how to find upper bounds on the number of pairs of pants of bounded length that only depend on the boundary length and the topology of the surface.
\vspace{1cm}

\section{Introduction}

In the last few decades, identities have played an integral part of the study of hyperbolic surfaces and their moduli spaces. They are generally equations which express a geometric quantity or surface invariant in terms of the lengths of a family of curves. For instance the McShane identity \cite{McShane} is a way of expressing the horocyclic boundary of a cusp in terms of the lengths of embedded pants. Around the same time, Basmajian \cite{Basmajian} proved an identity relating the boundary length of a surface with boundary to the set of lengths of orthogeodesics. The Bridgeman identity \cite{Bridgeman} used these same orthogeodesic lengths to express the volume of the unit tangent bundle. This same volume of the unit tangent bundle was decomposed by Luo and Tan in terms of the boundary lengths of embedded pants and one-holed tori. The Luo-Tan identity is the first of these identities that doesn't require the surface to have any cusp or geodesic boundary.

One interpretation of these identities is that they associate a measure to each element of the index set, and although these individual measures vary in terms of the geometry of the surface, their sum remains invariant. We investigate the measures of the Luo-Tan identity, and a refinement of the identity due to Hu and Tan. A precise version of the identities (and in particular a description of the index sets) will be given in the next section, but for reference we recall them here. The Luo-Tan identity states:
\begin{equation*}
\sum_{P\in \mathcal{P} }\varphi(P)+\sum_{T\in \mathcal{T}}\tau(T)=8\pi^2(g-1).
\end{equation*}
where the sums are taken over so-called properly embedded pairs of pants and one holed tori. The measures, $\varphi$ and $\tau$, are functions that depend explicitly on the geometries of $P$ or $T$. The Hu-Tan variation of the identity can be stated as follows:
\begin{equation*}
\sum_{P\in \mathcal{P} }\varphi(P)+\sum_{P\in \mathcal{I}}\eta(P)=8\pi^2(g-1),
\end{equation*}
where the sums are taken over properly and improperly embedded pants, $\varphi$ is the same as before and $\eta$ is a different function from $\varphi$ but which also depends explicitly on the geometry of $P$. 

Our main results are about analytic properties of the measures. We state the most striking (and useful) properties here, which concern the measures $\varphi$ and $\eta$. As they depend only on the geometry of the pants, they depend only on the boundary lengths of the pants. Hence $\varphi$ depends on three real parameters and $\eta$ only two as two of its boundary curves are of equal length. For practical reasons it is useful to consider, instead of length $\ell$, the parameter $t:= e^{-\ell/2}$. With these parameters, our results can be expressed as follows.

\begin{thm}\label{thm:main}
The functions $\varphi$ and $\eta$ are strictly increasing on $(0,1]^3$ and $(0,1]^2$, respectively, and satisfy
$$
\varphi(x,y,y) \leq \eta(x,y).
$$
Furthermore if we set $t:=\sqrt[3]{xyz}$ then
$$\varphi(x,y,z) \geq \varphi(t,t,t)> -24t^3\log(t)+24t^3 $$
for all $x,y,z\in(0,1]$.
\end{thm}

In particular this says that the measures are strictly decreasing with respect to boundary length. One might expect this as they necessarily converge to $0$ as the lengths increase (because there are infinitely many terms in the sum which adds up to something finite), but in fact there is no obvious geometric reason for this to hold infinitesimally and our proof is entirely analytic. As in Bridgeman's identity, the functions involve Rogers' dilogarithm function and have an intrinsic interest, but our original motivation for studying them was for possible applications. 

From our result, we are able to deduce a few corollaries. As a first application, we recover a well-known and useful theorem of Thurston's about dominating length spectra \cite{Thurston}. 

\begin{coro}[Thurston] If $X$ and $Y$ are marked and closed hyperbolic surfaces of genus $g$ that satisfy $\ell_X(\gamma) \geq \ell_Y(\gamma)$ for all simple closed geodesics $\gamma$, then $X=Y$. 
\end{coro}
The surfaces $X$ and $Y$ are points in Teichm\"uller space (the space of marked hyperbolic metrics) and Thurston used this result to deduce a positivity result for his asymmetric metric on Teichm\"uller space, related to Lipschitz maps between hyperbolic surfaces.

It should be noted that the same result for surfaces with cusps is easily deduced from McShane's identity. Indeed, the summands in the McShane identity are of the form $\frac{1}{e^{(\ell(\alpha)+\ell(\beta))/2}+1}$ and thus are obviously strictly decreasing in both $\ell(\alpha)$ and $\ell(\beta)$. A more general observation of this type can be found in the work of Charette and Goldman \cite{Charette-Goldman}.

As a second application, we count pants, and find an upper bound on the number of pants of total boundary length $L$ a surface of genus $g$ can have.
\begin{coro}\label{cor:pants}
A closed hyperbolic surface $X$ of genus $g$ has strictly less than 
$$ \frac{2\pi^2(g-1)e^{L/2}}{L+6}$$
embedded geodesic pants of total boundary length less than $L$. 
\end{coro}
This result is related to other results about curve counting. Of course, by the celebrated results of Mirzakhani \cite{Mirzakhani}, the number of pairs of pants grows asymptotically like $C_XL^{6g-6}$ where $C_X$ is a constant that depends on the surface, so it is far from optimal for large $L$. Nonetheless, the result above is an absolute upper bound that doesn't depend on the geometry of the surface. In particular it holds for all $L>0$, including relatively small $L$. A more directly related result is a result of Buser \cite{Buser} which says that a surface of genus $g$ has at most $(g-1)e^{L+6}$ primitive closed geodesics of length at most $L$. This result is used, among other things, to find upper bounds on the number of surfaces that can have the same length spectra are not  isometric. Also notice that Buser's result can be applied to find an upper bound on the number of pants of total length $L$, but the result is a lot weaker. In a nutshell, Buser's upper bound and the above corollary are related, but do not follow from one another.

One of the novelties of the Luo-Tan identity is that it {\it also} holds for closed surfaces, hence for simplicity we've stated our results in this context. However, with the usual caveats, they generalize without difficulty to surfaces with cusps. This can be seen either by applying the same methods, or by considering cusped surfaces as lying in the compactification of the underlying moduli space.

\section{The Luo-Tan identity and variations}

In this section we recall a precise formulation of the identity and rewrite it in a slightly different form more convenient for our purposes.

Let $X$ be a closed, orientable hyperbolic surface of genus $g\geq 2$. We will generally be thinking of $X$ as {\it marked}, hence as a point in Teichm\"uller space $\TT_g$, or if the marking is not essential, in moduli space $\MM_g$. (Marked in this setting can be thought of as knowing the names of all simple closed geodesics.) We shall be investigating different small complexity subsurfaces of $X$. The subsurfaces we consider are all either considered up to isotopy or equivalently, we consider their geodesic realizations (their boundary curves are simple closed geodesics). An (geodesic) embedded three-holed sphere $P\subset X$ (or pair of pants) is said to be properly embedded if its closure is embedded. (In other words, all three of its boundary curves are non-isotopic.) Otherwise its closure is an embedded one-holed torus and it is said to be improperly embedded.

With that in hand, the Luo-Tan identity \cite{Luo-Tan} states the following:
\begin{equation}\label{eq1}
\sum_{P\in \mathcal{P} }\varphi(P)+\sum_{T\in \mathcal{T}}\tau(T)=8\pi^2(g-1).
\end{equation}

The right hand side of the identity is the volume of unit tangent bundle. The left hand side has two index sets. 
The first ($\mathcal{P}$) is the set of properly embedded geodesic pants on $X$ whereas the second ($\mathcal{T}$) is the set of embedded geodesic one holed tori. 

The functions depend explicitly on the geometries of the pants and tori. Hence both can be made to depend on three real variables. We think of these functions as being measures on the set of pants and tori, where the measures sum up to full volume. 

By cutting a one holed torus along a simple closed geodesic, one obtains a pair of pants. Hence tori contain infinitely many distinct geodesic pants with embedded interior but, because of their boundary geodesics, their closures fail to be embedded. Extending the function for embedded pants to these {\it improperly} embedded pants leads to under counting and hence an inequality. Nonetheless, Hu and Tan \cite{Hu-Tan} found a way of decomposing the measure associated to a one holed torus as an infinite sum of measures associated to improperly embedded pants. Putting together the results leads to a new identity where the second summand set is on improperly embedded pants (the set of which we denote by $\mathcal{I}$):
\begin{equation}\label{eq2}
\sum_{P\in \mathcal{P} }\varphi(P)+\sum_{P\in \mathcal{I}}\eta(P)=8\pi^2(g-1),
\end{equation} 
in which $\varphi$ and $\eta$ are functions that depends on the geometry of $P$. 
\\

\underline{{\it The function $\varphi$ on embedded pairs of pants}}

We now describe the functions explicitly. Let $P$ be a pair of pants with geodesic boundaries $\gamma_1, \gamma_2, \gamma_3$ of lengths $\ell_1,\ell_2,\ell_3$. For $\{i,j,k\}=\{1,2,3\}$, let $m_i$ be the length of the shortest geodesic arc between $\gamma_j$ and $\gamma_k$ for $\{i,j,k\} = \{1,2,3\}$.\\

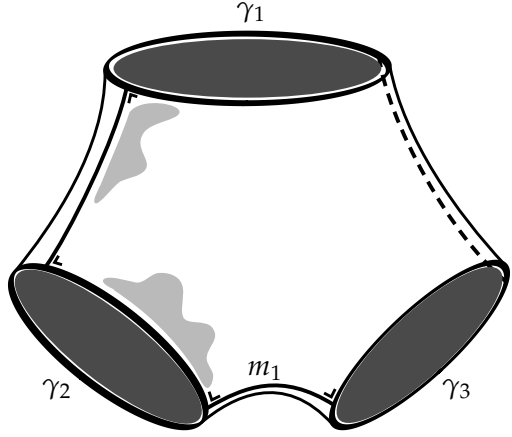
\begin{figure}[h]
\leavevmode \SetLabels
\L(.48*.96) $\gamma_1$\\%
\L(.31*.15) $\gamma_2$\\%
\L(.66*.15) $\gamma_3$\\%
\L(.49*.19) $m_1$\\
\endSetLabels
\begin{center}
\AffixLabels{\centerline{\epsfig{file =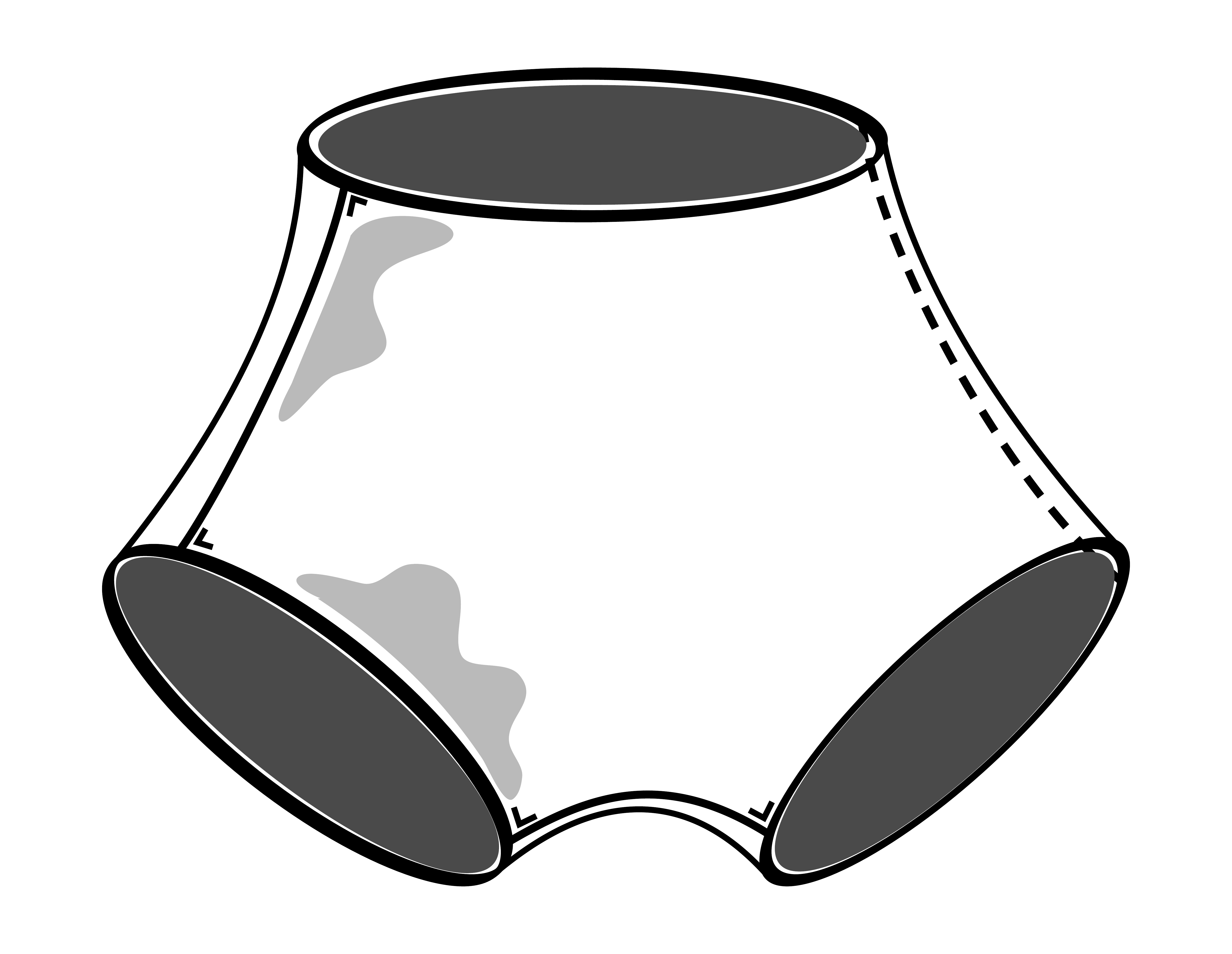,width=8cm,angle=0}}}
\vspace{-24pt}
\end{center}
\caption{A properly embedded pair of pants} \label{fig:pants}
\end{figure}

The function $\varphi$ applied to $P$ can now be expressed as: 
$$\varphi(P):=4\sum_{i\neq j}\bigg[2\mathcal{L}\bigg(\frac{1-x^2_i}{1-x^2_iy_j}\bigg)-2\mathcal{L}\bigg(\frac{1-y_j}{1-x^2_iy_j}\bigg)-\mathcal{L}(y_j)-\mathcal{L}\bigg(\frac{(1-y_j)^2x^2_i}{(1-x^2_i)^2y_j}\bigg)\bigg],$$
where $x_i=e^{-\ell_i/2}$ and $y_i=\tanh^2(m_i/2)$. Note that $x_i$ is monotonic decreasing in $\ell_i$.\\

One of our goals will be to study the variation of this function in terms of the lengths $\ell_i$. For that purpose, we shall express the function solely in terms of the $x_i$. By our definition of $y_1$:
$$y_1=\tanh^2(m_1/2)=\frac{\sinh^2(m_1/2)}{\cosh^2(m_1/2)}=\frac{(\cosh(m_1)-1)/2}{(\cosh(m_1)+1)/2}=\frac{\cosh(m_1)-1}{\cosh(m_1)+1}.$$ 
Using standard hyperbolic trigonometry we have
$$\cosh(m_1)=\frac{\cosh(\ell_1/2)+\cosh(\ell_2/2)\cosh(\ell_3/2)}{\sinh(\ell_2/2)\sinh(\ell_3/2)}=\frac{(x_1+\frac{1}{x_1})/2+(x_2+\frac{1}{x_2})(x_3+\frac{1}{x_3})/4}{(\frac{1}{x_2}-x_2)(\frac{1}{x_3}-x_3)/4}.$$
Then $$y_1=\frac{2(x_1+\frac{1}{x_1})+(x_2+\frac{1}{x_2})(x_3+\frac{1}{x_3})-(\frac{1}{x_2}-x_2)(\frac{1}{x_3}-x_3)}{2(x_1+\frac{1}{x_1})+(x_2+\frac{1}{x_2})(x_3+\frac{1}{x_3})+(\frac{1}{x_2}-x_2)(\frac{1}{x_3}-x_3)}=\frac{(x_1x_3+x_2)(x_1x_2+x_3)}{(x_2x_3+x_1)(1+x_1x_2x_3)}.$$
More generally, for $\{i,j,k\}=\{1,2,3\}$, we obtain:
$$y_j=\frac{(x_jx_k+x_i)(x_jx_i+x_k)}{(x_ix_k+x_j)(1+x_1x_2x_3)}.$$
Hence
$$\frac{1-x^2_i}{1-x^2_iy_j}=\frac{1-x^2_i}{1-x^2_i\frac{(x_jx_k+x_i)(x_jx_i+x_k)}{(x_ix_k+x_j)(1+x_1x_2x_3)}}=\frac{(x_ix_k+x_j)(1+x_1x_2x_3)}{x_j+x_i^2x_j+x_ix_k+x_ix_j^2x_k}.$$
Similarly:
$$\frac{1-y_j}{1-x^2_iy_j}=\frac{x_j(1-x^2_k)}{x_j+x_i^2x_j+x_ix_k+x_ix_j^2x_k},$$
and
$$\frac{(1-y_j)^2x^2_i}{(1-x^2_i)^2y_j}=\frac{x_i^2x_j^2(1-x_k^2)^2}{(x_k+x_ix_j)(x_i+x_jx_k)(x_j+x_ix_k)(1+x_1x_2x_3)}.$$
In terms of $x_1,x_2$ and $x_3$ we obtain: 
$$\varphi(x_1,x_2,x_3):=4\sum_{\{i,j,k\}=\{1,2,3\}}\bigg[2\mathcal{L}\bigg(\frac{(x_ix_k+x_j)(1+x_1x_2x_3)}{x_j+x_i^2x_j+x_ix_k+x_ix_j^2x_k}\bigg)$$
$$-2\mathcal{L}\bigg(\frac{x_j(1-x^2_k)}{x_j+x_i^2x_j+x_ix_k+x_ix_j^2x_k}\bigg)-
\mathcal{L}\bigg(\frac{(x_jx_k+x_i)(x_jx_i+x_k)}{(x_ix_k+x_j)(1+x_1x_2x_3)}\bigg)$$
$$-\mathcal{L}\bigg(\frac{x_i^2x_j^2(1-x_k^2)^2}{(x_k+x_ix_j)(x_i+x_jx_k)(x_j+x_ix_k)(1+x_1x_2x_3)}\bigg)\bigg].$$

\underline{{\it The function $\eta$ on improperly embedded pants}} 

Now let $T$ be a hyperbolic one-holed torus with boundary geodesic $\beta$ and let $\alpha$ be a non-peripheral simple closed geodesic of $T$. Let $h_\alpha$ be the length of the shortest simple orthogeodesic from $\beta$ to itself which is disjoint from $\alpha$. Let $p_\alpha$ denote the length of the pair of shortest simple orthogeodesics from $\alpha$ to $\beta$. Finally let $q_\alpha$ be 
the length of the shortest simple orthogeodesic from $\alpha$ to itself. 

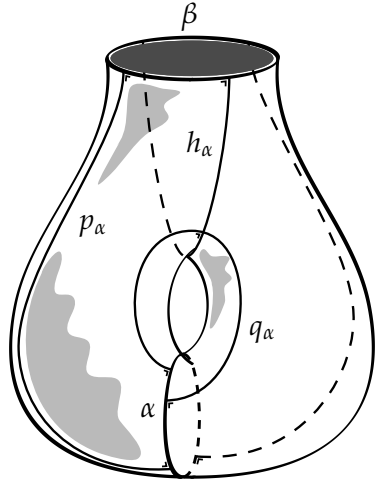
\begin{figure}[h]
\leavevmode \SetLabels
\L(.49*1.01) $\beta$\\%
\L(.455*.15) $\alpha$\\%
\L(.495*.73) $h_\alpha$\\%
\L(.4*.56) $p_\alpha$\\
\L(.55*.33) $q_\alpha$\\
\endSetLabels
\begin{center}
\AffixLabels{\centerline{\epsfig{file =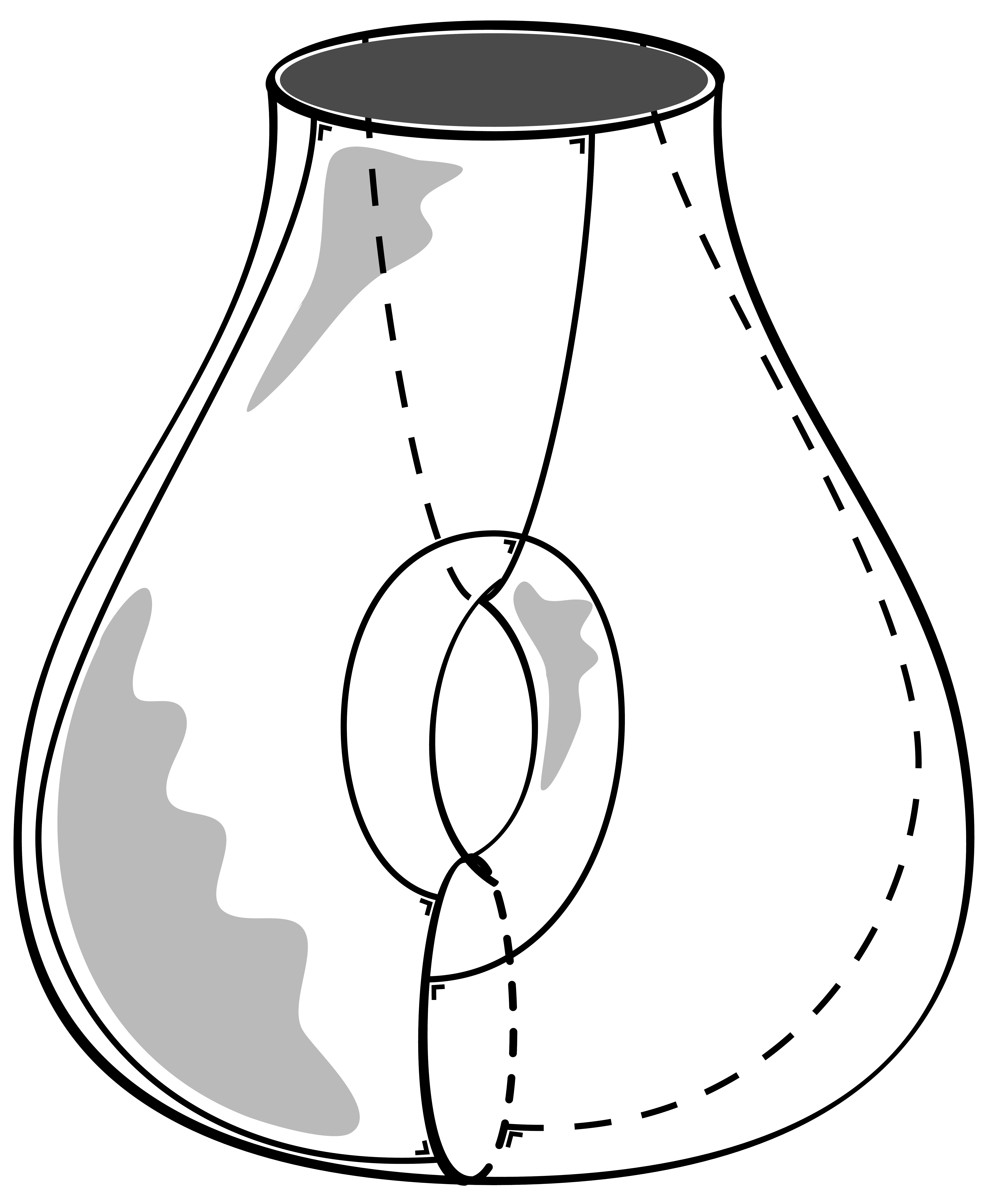,width=5cm,angle=0}}}
\vspace{-24pt}
\end{center}
\caption{An improperly embedded pair of pants} \label{fig:torus}
\end{figure}

Let $P$ be the improperly embedded pair of pants associated to $T$ by cutting $T$ along $\alpha$. Then the function $\eta$ is defined as:
$$\eta(P):=8\bigg[\mathcal{L}\bigg(\tanh^2\bigg(\frac{q_\alpha}{2}\bigg)\bigg)+2\mathcal{L}\bigg(\tanh^2\bigg(\frac{h_\alpha}{2}\bigg)\bigg)-\mathcal{L}\bigg(\sech^2\bigg(\frac{p_\alpha}{2}\bigg)\bigg)$$
$$-2\LA\bigg(e^{-\ell(\alpha)},\tanh^2\bigg(\frac{h_\alpha}{2}\bigg)\bigg)-2\LA\bigg(e^{-\frac{\ell(\beta)}{2}},\tanh^2\bigg(\frac{h_\alpha}{2}\bigg)\bigg)\bigg].$$
 Let $x:=e^{\frac{-\ell(\beta)}{2}}$ and $y:=e^{\frac{-\ell(\alpha)}{2}}$, we will express each term of $\eta(P)$ in term of $x$ and $y$.
 $$\tanh^2\bigg(\frac{q_\alpha}{2}\bigg)=\frac{\cosh(q_\alpha)-1}{\cosh(q_\alpha)+1}=\frac{\cosh(\ell(\beta)/2)+\cosh^2(\ell(\alpha)/2)-\sinh^2(\ell(\alpha)/2)}{\cosh(\ell(\beta)/2)+\cosh^2(\ell(\alpha)/2)+\sinh^2(\ell(\alpha)/2)}$$
 $$=\frac{x+1/x+2}{x+1/x+y^2+1/y^2}
 =\frac{(x+1)^2y^2}{(x+y^2)(x y^2+1)}.$$
 Similarly, $$\tanh^2\bigg(\frac{h_\alpha}{2}\bigg)=\frac{x+y^2}{xy^2+1}.$$
 and $$\sech^2\bigg(\frac{p_\alpha}{2}\bigg)=\frac{1}{\cosh^2\left(\frac{p_\alpha}{2}\right)}=\frac{1}{\sinh^2(\ell(\alpha)/2)\sinh^2(h_\alpha)}=\frac{(1-\tanh^2(\frac{h_\alpha}{2}))^2y^2}{(1-y^2)^2\tanh^2(\frac{h_\alpha}{2})}$$
 $$=\frac{(1-x)^2y^2}{(x+y^2)(x y^2+1)}.$$
 
Following \cite{Luo-Tan}, we define the lasso function $\LA$ as follows:
 $$\LA(a,b):=\mathcal{L}(b)+\mathcal{L}\bigg(\frac{1-b}{1-ab}\bigg)-\mathcal{L}\bigg(\frac{1-a}{1-ab}\bigg),$$ for $a,b \in (0,1)$.
 Hence $$\LA\bigg(e^{-\ell(\alpha)},\tanh^2\bigg(\frac{h_\alpha}{2}\bigg)\bigg)=\LA\bigg(y^2,\frac{x+y^2}{xy^2+1}\bigg)=\mathcal{L}\bigg(\frac{x+y^2}{xy^2+1}\bigg)+\mathcal{L}\bigg(\frac{1-x}{1+y^2}\bigg)-\mathcal{L}\bigg(\frac{xy^2+1}{y^2+1}\bigg),$$ and 
 $$\LA\bigg(e^{-\frac{\ell(\beta)}{2}},\tanh^2\bigg(\frac{h_\alpha}{2}\bigg)\bigg)=\LA\bigg(x,\frac{x+y^2}{xy^2+1}\bigg)\bigg)=\mathcal{L}\bigg(\frac{x+y^2}{xy^2+1}\bigg)+\mathcal{L}\bigg(\frac{1-y^2}{1+x}\bigg)-\mathcal{L}\bigg(\frac{xy^2+1}{x+1}\bigg).$$
 In terms of $x$ and $y$ we obtain:
 $$\eta(x,y)=8\mathcal{L}\bigg(\frac{(x+1)^2y^2}{(x+y^2)(xy^2+1)}\bigg)-8\mathcal{L}\bigg(\frac{(1-x)^2y^2}{(x+y^2)(xy^2+1)}\bigg)-16\mathcal{L}\bigg(\frac{1-x}{1+y^2}\bigg)$$
 $$+16\mathcal{L}\bigg(\frac{xy^2+1}{1+y^2}\bigg)-16\mathcal{L}\bigg(\frac{x+y^2}{xy^2+1}\bigg)-16\mathcal{L}\bigg(\frac{1-y^2}{1+x}\bigg)+16\mathcal{L}\bigg(\frac{xy^2+1}{1+x}\bigg).$$

Now that we have properly defined the functions $\varphi$ and $\eta$, we refer the reader back to our main result as stated in the introduction (Theorem \ref{thm:main}). Recall that if we express $\varphi(P)$ and $\eta(P)$ in terms of the boundary lengths, this theorem tells us that the functions are strictly {\it decreasing} in terms of these lengths. We defer the proof of Theorem \ref{thm:main} to the final section, and now concentrate on certain of its implications. 

\section{Dominating simple length spectra}

In this section we show how to deduce a theorem of Thurston's (Theorem 3.1 in \cite{Thurston}) from the monotonicity properties of the measures.

\begin{thm} If $X,Y \in \TT_{g}$ satisfy $\ell_X(\gamma) \geq \ell_Y(\gamma)$ for all simple closed geodesics $\gamma$, then $X=Y$. 
\end{thm}

\begin{proof}
Just for the purpose of this proof we think of the functions $\varphi$ and $\eta$ as being functions of the boundary lengths. In order not to introduce too much notation, we continue to call them $\varphi$ and $\eta$.

Now if $\ell_X(\gamma) \geq \ell_Y(\gamma)$ for all $\gamma$, then in particular, by monotonicity of the function $\varphi$, for any embedded pair of pants with boundary curves $\gamma_1,\gamma_2$ and $\gamma_3$:
$$
\varphi(\ell_X(\gamma_1), \ell_X(\gamma_2), \ell_X(\gamma_3)) \leq \varphi(\ell_Y(\gamma_1), \ell_Y(\gamma_2), \ell_Y(\gamma_3))
$$
with equality if and only if the lengths are all equal. Similarly, for an improperly embedded pair of pants with boundary curve $\beta$ and interior simple closed curve $\alpha$, we have, by monotonicity of the function $\eta$:
$$
\eta(\ell_X(\beta), \ell_X(\alpha))\leq \eta(\ell_Y(\beta), \ell_Y(\alpha))
$$
with equality if and only if the lengths are equal. As, by Equation \ref{eq2}, the sums of these functions, summed over all possible properly and improperly embedded pants, are equal for both $X$ and $Y$, it follows that each summand is equal.

Now as every simple closed geodesic belongs to certain pairs of pants, either properly or improperly embedded, the result follows by rigidity of the marked simple length spectrum.
\end{proof}

\section{Counting pants}

We now focus on counting the number of pants (embedded or improperly embedded) of boundary length less than $L$ on any surface (of genus $g$ with $n$ cusps). 

Fix a hyperbolic closed surface $X$ and let $\mathcal{Y}(X)$ be the set of isotopy classes of geodesic pants on $X$ (so the union of properly $\mathcal{P}(X)$ and improperly embedded pants $\mathcal{I}(X)$). The Hu-Tan variation of the identity (Equation \ref{eq2}) allows us to associate to each pair of pants $P$ a measure (either $\varphi$ or $\eta$ depending on whether it is properly or improperly embedded), the sum of which adds up to the volume of the unit tangent bundle. The inequalities from Theorem \ref{thm:main} will then allow us  to bound the number of pants.

\begin{thm}\label{thm:pants}
For a surface $X$ we set
$${\mathrm{NP}}_X(L) = \sharp \left\{Y \in \mathcal{Y}(X) | \ell(\partial{Y}) \leq L\right\}$$
to be the number of pants of total boundary length less than $L$. Then any $X\in \MM_g$ satisfies
$$\mathrm{NP}_X(L) < \frac{2\pi^2(g-1)e^{L/2}}{L+6}.$$
\end{thm}
\begin{proof}
Let $Y \in \mathcal{Y}(X)$. Suppose $Y$ is properly embedded, and let $\ell_1, \ell_2$ and $\ell_3$ be its boundary lengths and $L$ their sum. Notice that
$$
e^{-L/6} = \sqrt[3]{e^{-\frac{\ell_1}{2}}e^{-\frac{\ell_2}{2}}e^{-\frac{\ell_3}{2}}}
$$
and thus by Theorem \ref{thm:main} we have
$$
\varphi(Y) = \varphi\left(e^{-\ell_1/2},e^{-\ell_2/2},e^{-\ell_3/2}\right) \geq \varphi\left(e^{-\frac{L}{6}},e^{-\frac{L}{6}},e^{-\frac{L}{6}}\right).
$$
Similarly, if $Y$ is improperly embedded with boundary lengths $\ell_1, \ell_2$ and $\ell_2$, and $L$  their sum, we have
$$
\eta(Y) = \eta\left(e^{-\ell_1/2},e^{-\ell_2/2}\right) \geq \varphi\left(e^{-\frac{L}{6}},e^{-\frac{L}{6}},e^{-\frac{L}{6}}\right)
$$
where the last inequality is again from Theorem \ref{thm:main}. Now by Theorem \ref{thm:main} again, the measure associated to any $Y$ of total boundary length $L$ is greater than
$$
\varphi\left(e^{-\frac{L}{6}},e^{-\frac{L}{6}},e^{-\frac{L}{6}}\right) > 24\, e^{-L/2} \frac{L}{6} + 24 \, e^{-L/2}= 4\frac{L+6}{e^{L/2}}.
$$
Now as the total sum of all the measures is equal to $8 \pi^2 (g-1)$, we have
$$
\mathrm{NP}_X(L) < 8 \pi^2 (g-1) \frac{1}{4\frac{L+6}{e^{-L/2}}}= \frac{2\pi^2(g-1)e^{L/2}}{L+6}
$$
as desired.
\end{proof}

\section{Behavior of the measures}

This is the main technical part of the paper, where we show the measures satisfy the properties we previously claimed. Several of the intermediate claims, although they are ultimately purely calculus, are in fact quite technical. They can (and have been) checked by formal computational software. 

The following lemma is part of Theorem \ref{thm:main}.

\begin{lem}\label{a}
$\varphi(x,x,x) > -24x^3\log(x)+24x^3$ for all $x\in (0,1]$.
 \end{lem}
 \begin{proof}
 Consider the function $f(x):=\varphi(x,x,x) + 24x^3\log(x)-24x^3$. This function is continuous on $(0,1]$, so it is enough to prove that $f$ is strictly increasing on $(0,1)$. Indeed, after taking the derivative of $f$ and manipulating terms of $f'$ reasonably, we obtain:
 $$f'(x)=(\varphi(x,x,x))' -48 x^2 + 72 x^2 \log(x)=m\Bigg(a \log(1 - x) + 
 b \log(x) 
 + c \log(1 + x^3) $$
 $$ 
 -d \log\left(1 - x + x^2\right) + 
 h \log\left(\frac{1 - x + x^2}{1 - x}\right)-2(1 + x^3) (1 - x) x^3\Bigg),$$
 where $$m:=\frac{24}{(1 - x) x (1 + x^3)},\,\,\,\, a:=x (1 + x) (1 - x)^2,\,\,\,\, b:=3 (1 - x) x^6 ,\,\,\,\, c:=(1 + x^3) (1 - x),$$
 $$ d:=x (1 + x),\,\,\,\, h:=x^2 (1 - x^2) .$$
Note that the following Taylor series for $\log(t)$ around $1$ is valid for $t \in (0,2]$: $$\log(t)=(t-1)-\frac{(t-1)^2}{2}+\frac{(t-1)^3}{3}-...=\sum_{k=1}^{\infty}\frac{(-1)^{k-1}(t-1)^k}{k}.$$
We observe that for all $t\in(0,1]$, the terms of the Taylor series are negative which implies that:
$$\log(t)\leq (t-1)-\frac{(t-1)^2}{2}.$$
Therefore:
$$f'(x)\geq m\Bigg(a \log(1 - x) + 
 b \left(1-\frac{1}{x}\right) 
 + c \left(1-\frac{1}{1 + x^3}\right) $$
 $$ -d \left((1 - x + x^2-1)-\frac{(1 - x + x^2-1)^2}{2}\right)+ h \left(1-\frac{1 - x}{1 - x + x^2}\right)-2(1 + x^3) (1- x)x^3\Bigg).$$
 Simplifying the right hand side of the above inequality we get:
 $$f'(x)\geq \frac{12 x (2 - x + 3 x^2 - 4 x^3 + 9 x^4 - 9 x^5 + 2 x^6)}{(1 + x) (1 - 
 x + x^2)^2} + \frac{24 (1 - x) \log(1 - x)}{1 - x + x^2}$$
 $$= \frac{24 (1 - x)}{1 - x + x^2}\left(\frac{x (2 - x + 3 x^2 - 4 x^3 + 9 x^4 - 9 x^5 + 2 x^6)}{2(1 + x) (1-x)(1 - 
 x + x^2)} + \log(1 - x)\right).$$
We now set $$g(x):=\frac{ x (2 - x + 3 x^2 - 4 x^3 + 9 x^4 - 9 x^5 + 2 x^6)}{2(1 + x) (1-x)(1 - x + x^2)} + \log(1 - x)$$
and so
$$g'(x)=\frac{x^2 (5 - 15 x + 34 x^2 - 46 x^3 + 28 x^4 + 4 x^5 - 18 x^6 + 13 x^7 - 
 3 x^8)}{(1 - x)^2 (1 + x)^2 (1 - x + x^2)^2}$$
 $$=\frac{x^2 (2 x^8 + 10 x^7 (1 - x) + (1 - x)^2 (5 - 5 x + 19 x^2 - 3 x^3 + 
 3 x^4 + 13 x^5 + 5 x^6))}{(1 - x)^2 (1 + x)^2 (1 - x + x^2)^2}>0,$$ for all $x\in(0,1)$. Therefore, $g(x)> g(0)=0$, for all $x \in (0,1).$

In particular, $f'(x)> 0$, for all $x \in (0,1)$ and thus
$$f(x) > \lim_{x\to0} \left(\varphi(x,x,x)+24x^3\log(x)-24x^3\right) = 0$$
which completes the proof.
 \end{proof}

We now  prove the monotonicity of $\varphi$. Since $\varphi$ is a symmetric function, it suffices to show:
\begin{lem} \label{d}
$$\partial_{x_1}{\varphi}(x_1,x_2,x_3)>0$$ for all $x_1,x_2,x_3 \in (0,1)$.
\end{lem}
\begin{proof} Taking the partial derivative of $\varphi$ with respect to the variable $x_1$, and by standard simplifications, one obtains:
$${\partial_{x_1}\varphi}(x_1,x_2,x_3)=\sum^3_{i=1}\bigg[a_i\log(x_i)+b_i\log(1-x_i^2)\bigg]+c_1\log(x_1+x_2x_3)$$
$$+c_2\log(x_2+x_1x_3)+c_3\log(x_3+x_1x_2) 
+M\log(1+x_1x_2x_3),$$
where
$$a_1:=-\frac{16(x_1+x_2x_3)}{(1-x_1^2)(1+x_1x_2x_3)}, \,\,\,\, a_2=a_3:=-\frac{16x_2x_3}{1+x_1x_2x_3},$$ $$b_1:=-\frac{8x_2x_3(1+x_1^2+2x_1x_2x_3)}{x_1(x_1+x_2x_3)(1+x_1x_2x_3)},$$
$$b_2:=\frac{8x_3(1-x_2^2)}{(x_2+x_1x_3)(1+x_1x_2x_3)}, \,\,\,\,b_3:=\frac{8x_2(1-x_3^2)}{(x_3+x_1x_2)(1+x_1x_2x_3)},$$
$$c_1:=\frac{16x_1+8x_2x_3+8x_1^2x_2x_3}{(1-x_1^2)(1+x_1x_2x_3)}, \,\,\,\,c_2=c_3:=\frac{8x_2x_3}{1+x_1x_2x_3},$$
{\footnotesize{$$M:=-\frac{16x_1^4(x_2^2+x_3^2+x_2^2x_3^2)+8x_1^3x_2x_3(4+x_1^2+3x_2^2+3x_3^2)+8x_2^2x_3^2(4x_1^2-2)-8x_1x_2x_3(1+x_2^2+x_3^2)}{x_1(1-x_1^2)(x_1+x_2x_3)(x_2+x_1x_3)(x_3+x_1x_2)}.$$\\}}
Again, by some standard manipulations, we can express $\partial_{x_1}{\varphi}(x_1,x_2,x_3)$ as follows:
$$\partial_{x_1}{\varphi}(x_1,x_2,x_3)=\frac{8x_2x_3}{1+x_1x_2x_3}\log\bigg(\bigg(\frac{x_1}{x_2x_3}+1\bigg)\bigg(\frac{x_2}{x_1x_3}+1\bigg)\bigg(\frac{x_3}{x_1x_2}+1\bigg)\bigg)
+b_1\log(1-x_1^2)$$
$$+b_2\log(1-x_2^2)+b_3\log(1-x^2_3)+\frac{16x_1}{1-x_1^2}\log\bigg(1+\frac{x_2x_3}{x_1}\bigg)+M\log(1+x_1x_2x_3).$$

Note that for all $x_1,x_2, x_3 \in (0,1)$, $$b_2\log(1-x_2^2)+b_3\log(1-x_3^2)\geq b_2.\frac{-x_2^2}{1-x_2^2}+b_3.\frac{-x_3^2}{1-x_3^2}=\frac{-8x_2x_3}{1+x_1x_2x_3}.\bigg(\frac{x_2}{x_2+x_1x_3}+\frac{x_3}{x_3+x_1x_2}\bigg)$$ and
$$\frac{16x_1}{1-x_1^2}\log\bigg(1+\frac{x_2x_3}{x_1}\bigg)+M\log(1+x_1x_2x_3)\geq \frac{16x_1}{1-x_1^2}\log(1+x_1x_2x_3)+M\log(1+x_1x_2x_3)$$
$$=\frac{8x_2x_3(2x_2x_3+x_1x_2^2+x_1x_3^2+x_1-x_1^3)}{x_1(x_1+x_2x_3)(x_3+x_1x_3)(x_3+x_1x_2)}
\log(1+x_1x_2x_3)> 0.$$
Hence, $${\partial_{x_1}\varphi}(x_1,x_2,x_3)> \frac{8x_2x_3}{1+x_1x_2x_3}\log\bigg(\bigg(\frac{x_1}{x_2x_3}+1\bigg)\bigg(\frac{x_2}{x_1x_3}+1\bigg)\bigg(\frac{x_3}{x_1x_2}+1\bigg)\bigg)$$
$$\,\,\,\,\,\,\,\,\,\,\,\,\,\,\,\,\,\,\,\,\,\,\,\,\,\,\,\,\,\,\,\,\,\,\,\,\,\,\,\,\,\,\,\,\,-\frac{8x_2x_3}{1+x_1x_2x_3}.\bigg(\frac{x_2}{x_2+x_1x_3}+\frac{x_3}{x_3+x_1x_2}\bigg)+b_1\log(1-x_1^2)$$
 $$= \frac{8x_2x_3}{1+x_1x_2x_3}\bigg[\log\bigg(\frac{x_2}{x_1x_3}+1\bigg)-\frac{x_2}{x_2+x_1x_3}+\log\bigg(\frac{x_3}{x_1x_2}+1\bigg)-\frac{x_3}{x_3+x_1x_2}\bigg]$$
$$+ \frac{8x_2x_3}{1+x_1x_2x_3}\log\bigg(\frac{x_1}{x_2x_3}+1\bigg)+b_1\log(1-x_1^2).$$
Note that, $\log(1+x) \geq \frac{x}{1+x}$ for all $x>-1$. Therefore, $$\log\bigg(\frac{x_2}{x_1x_3}+1\bigg)\geq\frac{x_2}{x_2+x_1x_3},$$ and$$\log\bigg(\frac{x_3}{x_1x_2}+1\bigg)\geq\frac{x_3}{x_3+x_1x_2}.$$
And thus
$$\partial_{x_1}{\varphi}(x_1,x_2,x_3)> \frac{8x_2x_3}{1+x_1x_2x_3}\log\bigg(\frac{x_1}{x_2x_3}+1\bigg)+b_1\log(1-x_1^2)>0.$$
 \end{proof}

The next lemma is about the monotonicity of $\eta$: 
 \begin{lem}\label{e}
 The function $\eta$ satisfies
$$\partial{\eta}_x(x,y)>0 \text{ and }\partial{\eta}_y(x,y)>0 \text{ for all } x,y \in (0,1).$$
 \end{lem}
 \begin{proof}
 Through some standard manipulations, we can express $\partial{\eta}_x(x,y)$ in the following form:
 $$\partial{\eta}_x(x,y)=M\Bigg(A\log\left(\frac{1}{1-x}\right)+B\log(1-y^2)+C\log\left(\frac{1+xy^2}{y}\right)+D\log\left(1+xy^2\right)$$
 $$+E\log\left(\frac{x+y^2}{x+x^2y^2}\right)\Bigg),$$ where $A:=(1 - x) y^2 (1 + x^2 + 2 x y^2), \,\,\,\,B:=(1 - x) x (1 - y^4), \,\,\,\,C:=2 x y^2 (1 - x) (x + y^2), \\D:=(1 - x)^2 (1 + x) y^2, \,\,\,\,E:=x (1 + y^2) (x + y^2),$ and $$M:=\frac{8}{(1 - x) x (x + y^2) (1 + x y^2)}.$$
 Note that, $\log(t) \geq 1-\frac{1}{t}$ for all $t>0$. Therefore, for all $x,y\in(0,1)$, we have:
 $$\partial{\eta}_x(x,y)\geq M\Bigg(A.x+B.\left(\frac{-y^2}{1-y^2}\right)+C.\left(\frac{1+xy^2-y}{1+xy^2}\right)+D.\left(\frac{xy^2}{1+xy^2}\right)+E.\left(\frac{y^2(1-x^2)}{x+y^2}\right)\Bigg)$$
$$=\frac{8 y^2 (1 + x + x^2 + y^2 + 4 x y^2 + 2 x^2 y^2 + x^3 y^2+ 2 x y^4 + 3 x^2 y^4+2(1-y)(x+y^2))}{(x + y^2) (1 + x y^2)^2}.$$
 Hence
 $$\partial{\eta}_x(x,y)>0,$$ for all $x,y\in(0,1)$.
 
Now we proceed to the second inequality of this lemma. The quantity $\partial{\eta}_y(x,y)$ can be expressed as
 $$\partial{\eta}_y(x,y)=m\left(a\log\left(1-x\right)+b\log\left(\frac{1}{x}\right)+c\log\left(\frac{x+y^2}{y^2(1+xy^2)}\right)+d\log\left(\frac{1+xy^2}{1-y^2}\right)\right),$$ where $$a:=(1 - x^2) y^2 (1 - y^2), \,\,\,\,b:=x y^2 (1 - y^2) (x + y^2), \,\,\,\,c:=(1 + x) y^2 (x + y^2), $$ $$ d:=x (1 - y^2) (1 + 2 x y^2 + y^4),$$ and $$m:=\frac{16}{y (1 - y^2) (x + y^2) (1 + x y^2)}.$$\\
 Similarly, for all $x,y\in(0,1)$, we have:
 $$\partial{\eta}_y(x,y)\geq m\Bigg(a\left(1-\frac{1}{1-x}\right)+b\left(1-x\right)+c\left(1-\frac{y^2(1+xy^2)}{x+y^2}\right)+d\left(1-\frac{1-y^2}{1+xy^2}\right)\Bigg)$$
$$=\frac{16xy (1 + 2 x - x^2 + 2 y^2 + 2 x y^2 + 3 x^2 y^2 - x^3 y^2 + y^4 + 3 x y^4)}{(x + y^2) (1 + x y^2)^2}.$$
 Hence
 $$\partial{\eta}_y(x,y)>0,$$ for all $x,y\in(0,1)$.
 \end{proof}
The following lemma is an essential step in our inequalities.

\begin{lem}\label{c}The function $\varphi$ satisfies
$$\varphi(x,y,z)\geq \varphi(x,\sqrt{yz},\sqrt{yz}),$$ for all $x,y,z\in(0,1]$. Furthermore, the function $\varphi(x,y,z)-\varphi(x,\sqrt{yz},\sqrt{yz})$ is monotone increasing with respect to $x$.
\end{lem}

\begin{proof} The derivative with respect to $x$ of the function $\varphi(x,y,z)-\varphi(x,\sqrt{yz},\sqrt{yz})$ is of the following form:
 $$\partial_{x}{\varphi(x,y,z)}-\partial_{x}{\varphi(x,\sqrt{yz},\sqrt{yz})}=8(A+B),$$ where $$A:=\frac{(1 - y^2) z \log(1 - y^2)}{(y + x z) (1 + x y z)} + \frac{y (1 - z^2) \log(1 - z^2)}{(x y + z) (1 + x y z)} - \frac{2 (1 - y z) \log(1 - y z)}{(1 + x) (1 + x y z)},$$ $$B:=-\frac{ y z \log(y)}{1 + x y z} - \frac{ y z \log(z)}{1 + x y z} -\frac{2 y z \log(1 + x)}{1 + x y z} + \frac{y z \log(x y + z)}{1 + x y z} +\frac{y z \log(y + x z)}{1 + x y z}$$
 $$- \frac{(1 - x) (y - z)^2 \log(1 + x y z)}{(1 + x) (x y + z) (y + x z)}.$$
 If we can show that $\partial_{x}{\varphi(x,y,z)}-\partial_{x}{\varphi(x,\sqrt{yz},\sqrt{yz})}\geq 0$ for all $x,y,z \in(0,1)$, then it will imply that:
 $$\varphi(x,y,z)-\varphi(x,\sqrt{yz},\sqrt{yz})\geq \varphi(0,y,z)-\varphi(0,\sqrt{yz},\sqrt{yz})=0.$$
 
Our aim will be to show that both $A$ and $B$ are non-negative for all $x,y,z \in (0,1)$. As
 $$A.(1+xyz)(xy+z)(xz+y)(1+x)= x^2h_2(y,z)+(x-x^2)h_1(y,z)+(1-x^2)h_0(y,z),$$ where
 {\footnotesize{$$h_2(y,z):=2z(y+z)(1 - y^2)\log(1 - y^2)+ 2y(y+z)(1 - z^2)\log(1 - z^2) - 2 (y+z)^2(1 - y z) \log(1 - y z), $$}}
 {\footnotesize{$$h_1(y,z):=z(y + z)(1 - y^2) \log(1 - y^2) +y(y + z) (1 - z^2)\log(1 - z^2) - 2(y^2+z^2)(1-yz) \log(1 - y z),$$}}
 and {\footnotesize{$$h_0(y,z):=z^2(1 - y^2) \log(1 - y^2)+ y^2 (1 - z^2)\log(1 - z^2) - 2yz(1 - y z) \log(1 - y z),$$}}
the non-negativity of $A$ is implied from the following:
 \begin{claim}\label{g}
 $$h_0(y,z)\geq0,\,\,\,\,h_1(y,z)\geq0, \text{ and }h_2(y,z)\geq0$$ for all $ 0<y,z<1$.
 \end{claim}
 \begin{proof}[Proof of claim]Note that:
 $$\frac{h_0(y,z)}{y^2z^2}=\frac{1 - y^2}{y^2} \log(1 - y^2)+ \frac{1 - z^2}{z^2} \log(1 - z^2) - 2\frac{1 - y z}{yz} \log(1 - y z) .$$
 We consider the following function:
 $$g(t):=\frac{(1-e^t)\log(1-e^t)}{e^t},$$ where $t<0$. Then $$g''(t)=\frac{1}{1-e^t}+\frac{\log(
 1 - e^t)}{e^t},$$
 which is easily checked to be positive for all $t<0$. Hence $g$ is convex on its domain. Therefore, for all negative numbers $t_1$ and $t_2$, we have:
$$g(t_1)+g(t_2) \geq 2g\left(\frac{t_1+t_2}{2}\right).$$
By substituting $t_1, t_2$ by $\log(y^2), \log(z^2)$ respectively, we obtain:
$$\frac{1 - y^2}{y^2} \log(1 - y^2)+ \frac{1 - z^2}{z^2} \log(1 - z^2) \geq 2\frac{1 - y z}{yz} \log(1 - y z) .$$
This implies that $h_0(y,z)\geq0$ for all $ 0<y,z<1$.

Now we prove that $h_2(y,z)\geq 0$. Indeed,
$$\frac{h_2(y,z)}{2(y+z)}=z(1 - y^2) [\log(1 - y^2)- \log(
 1 - y z) ]+ y (1 - z^2) [\log(1 - z^2)-\log(
 1 - y z) ]$$
 $$=z(1 - y^2) \log\left(\frac{1 - y^2}{1 - y z}\right)+ y (1 - z^2) \log\left(\frac{1 - z^2}{1 - y z}\right)$$
 $$\geq z(1 - y^2) \left(1-\frac{1 - yz}{1 - y^2}\right)+ y (1 - z^2)\left(1-\frac{1 - yz}{1 - z^2}\right)= 0,$$
 for all $y,z \in (0,1)$.\\

Lastly, $h_1(y,z)$ is non-negative because of the following:
$$\frac{h_1(y,z)}{y+z}=\left[z(1-y^2)\log(1-y^2)+y(1-z^2)\log(1-z^2)\right]-2\frac{(y^2+z^2)(1-yz)}{y+z}\log(1-yz)$$ $$=\left[\frac{h_2(y,z)}{2(y+z)}+(z(1-y^2)+y(1-z^2))\log(1-yz)\right]-2\frac{(y^2+z^2)(1-yz)}{y+z}\log(1-yz)$$
 $$=\frac{h_2(y,z)}{2(y+z)}-\frac{(y-z)^2(1-yz)}{y+z}\log(1-yz) \geq 0,$$ for all $y,z \in (0,1)$.
 This completes the proof of Claim \ref{g}. 
 \end{proof}
Finally, we prove the non-negativity of $B$ as follows:
$$B=\frac{yz}{1+xyz}\log\left(\frac{(xy+z)(xz+y)}{yz(1+x)^2} \right)- \frac{(1 - x) (y - z)^2 \log(1 + x y z)}{(1 + x) (x y + z) (y + x z)}$$ $$\geq \frac{yz}{1+xyz}\left(1-\frac{yz(1+x)^2}{(xy+z)(xz+y)} \right)- \frac{(1 - x) (y - z)^2 xyz}{(1 + x) (x y + z) (y + x z)}$$ 
$$=\frac{x^2 y (y - z)^2 z (2 -yz + xy z)}{(1 + x) (x y + z) (y + 
 x z) (1 + x y z)}\geq 0, \text{ for all } x,y,z \in(0,1).$$
\end{proof}
\begin{remk}
The previous proof tells us that
$$\frac{\partial{\varphi}}{\partial{x}}(x,y,z)\geq \frac{\partial{\varphi}}{\partial{x}}(x,\sqrt{yz},\sqrt{yz}).$$ Hence, Lemma \ref{d} also follows from the following simpler inequality which contains only two variables:
$$\frac{\partial{\varphi}}{\partial{x}}(x,y,y)>0.$$
\end{remk}
\begin{lem}
The function $\varphi$ satisfies
$$\varphi(x,y,z)\geq \varphi(\sqrt[3]{xyz},\sqrt[3]{xyz},\sqrt[3]{xyz}),$$ for all $x,y,z\in(0,1]$
\end{lem}
\begin{proof}
By applying Lemma \ref{c} two times, we obtain:
\begin{equation}\label{cc}
\varphi(x,y,z)\geq \varphi(x,\sqrt{yz},\sqrt{yz})\geq \varphi\left(\sqrt{x\sqrt{yz}},\sqrt{x\sqrt{yz}},\sqrt{yz}\right)
\end{equation}
We define a function $f$ from $(0,1]^3$ to $(0,1]^3$ as follows:
$$f(x,y,z):=(x^{\frac{1}{2}}y^{\frac{1}{4}}z^{\frac{1}{4}},x^{\frac{1}{2}}y^{\frac{1}{4}}z^{\frac{1}{4}},y^{\frac{1}{2}}z^{\frac{1}{2}}),$$
then from (\ref{cc}), we have a monotonically decreasing sequence:
\begin{equation}\label{vv}
\varphi(x,y,z) \geq \varphi(f(x,y,z)) \geq \varphi(f^2(x,y,z))\geq ... \geq \varphi(f^n(x,y,z)) \geq ...
\end{equation}
By induction, we can show that:
$$f^n(x,y,z)=(x^{a_n}y^{b_n}z^{b_n},x^{a_n}y^{b_n}z^{b_n},x^{2b_n-a_n}y^{a_n}z^{a_n})$$
for all $n \in \mathbb{N}$, in which $a_n=\frac{1}{3}+\frac{2}{3}\left(\frac{1}{4}\right)^n$, and $b_n=\frac{1}{3}-\frac{1}{3}\left(\frac{1}{4}\right)^n$. Hence 
$$\lim_{n \to \infty}(f^n(x,y,z))=(x^{\frac{1}{3}}y^{\frac{1}{3}}z^{\frac{1}{3}},x^{\frac{1}{3}}y^{\frac{1}{3}}z^{\frac{1}{3}},x^{\frac{1}{3}}y^{\frac{1}{3}}z^{\frac{1}{3}}).$$
Therefore, from (\ref{vv}) and the continuity of the function $\varphi$ on its domain, we have:
$$\varphi(x,y,z) \geq \lim_{n \to \infty}\varphi(f^n(x,y,z))=\varphi(\lim_{n \to \infty}(f^n(x,y,z))=\varphi(\sqrt[3]{xyz},\sqrt[3]{xyz},\sqrt[3]{xyz}).$$

\end{proof}
The following lemma is a relation between the two functions $\varphi$ and $\eta$:
\begin{lem}\label{z}The functions $\varphi$ and $\eta$ satisfy:
$$\eta(x,y)\geq \varphi(x,y,y),$$ for all $x,y,z\in(0,1]$. Furthermore, the function $\eta(x,y)-\varphi(x,y,y)$ is monotone increasing with respect to $x$.
\end{lem}
\begin{proof}
We will prove that:
 $$\partial_x{\eta(x,y)}-\partial_x{\varphi(x,y,y)}>0,$$ for all $x,y \in(0,1)$. Indeed, 
 $$\partial_x{\eta(x,y)}-\partial_x{\varphi(x,y,y)}=m\Bigg(a \log(1 - y^2)+b\log\left(\frac{1}{1 + x y^2}\right)+c \log\left(\frac{1 + x}{1 + x y^2}\right)$$
 $$+d\log\left(\frac{x + y^2}{1 + x y^2}\right)+h\log\left(\frac{x + y^2}{x (1 + x y^2)}\right) \Bigg),$$ where $a:=(1 - x) x (1 - y^2)^2,\,\,\,\,b:=(1 - x) x (1 - y^2)^2,\,\,\,\,c:=(1 - x) (1 + x)^2 y^2, $ $$d:=x (1 + x) y^2 (x + y^2),\,\,\,\,h:=x (1 - y^2) (x + y^2),$$ and
 $$m:=\frac{8}{x (1 + x) (x + y^2) (1 + x y^2)}$$\\
 Note that, $\log(t) \geq 1-\frac{1}{t}$ for all $t>0$. Therefore, for all $x,y\in(0,1)$, we have:
 $$\partial_x{\eta(x,y)}-\partial_x{\varphi(x,y,y)}\geq m\Bigg(a \left(1-\frac{1}{1 - y^2}\right)+b\left(1-(1 + x y^2)\right)+c \left(1-\frac{1 + x y^2}{1 + x}\right)$$
  $$+d\left(1-\frac{1 + x y^2}{x + y^2}\right)+h\left(1-\frac{x (1 + x y^2)}{x + y^2}\right) \Bigg)$$
  $$=\frac{8 (1 - x) x y^4 (1 - y^2)}{(1 + x) (x + y^2) (1 + x y^2)}> 0. $$
  This implies that:  $${\eta(x,y)}-{\varphi(x,y,y)} \geq \eta(0,y)-\varphi(0,y,y)=0,$$
  for all $x,y \in (0,1]$.
\end{proof}

This completes the proofs of the technical results in this note. We end with an example in a similar vein, which can be obtained by using the same methods, and which illustrates that the function $\varphi$ has a wealth of yet unexplored properties. 
\begin{lem}\label{zz}The function $\varphi$ satisfies:
$$\varphi(x,yz,1)\geq \varphi(x,y,z),$$ for all $x,y,z\in(0,1]$. Furthermore, the function $\varphi(x,yz,1)-\varphi(x,y,z)$ is monotone increasing with respect to $x$.
\end{lem}

{\it Addresses:}\\
Department of Mathematics, University of Luxembourg, Esch-sur-Alzette, Luxembourg\\
Department of Mathematics, National University of Singapore, Singapore

{\it Emails:}\\
minh.doan@uni.lu\\
hugo.parlier@uni.lu\\
mattansp@nus.edu.sg

\end{document}